\documentclass[smallextended]{svjour3}    
\smartqed  

\usepackage{graphicx}
\usepackage [latin1]{inputenc}
\usepackage{amstext}
 \usepackage{amsthm}
\usepackage{amsfonts}
\usepackage{amsmath}

\theoremstyle{plain}
\ifx\thechapter\undefined
\newtheorem{thm}{\protect\theoremname}
\else
\newtheorem{thm}{\protect\theoremname}[chapter]
\fi
\theoremstyle{plain}
\ifx\thechapter\undefined
\newtheorem{lem}{\protect\lemmaname}
\else
\newtheorem{lem}{\protect\lemmaname}[chapter]
\fi
\providecommand{\lemmaname}{Lemma}
\providecommand{\theoremname}{Theorem}

\begin{document}

\title{Summation of certain trigonometric series with logarithmic coefficients}
\author{Rufus Boyack}
\institute{Rufus Boyack \at
              Department of Physics $\&$ Theoretical Physics Institute, University
of Alberta, Edmonton, Alberta T6G 2E1, Canada \\
          \email{boyack@ualberta.ca}}

\date{Received: date / Accepted: date}

\maketitle

\begin{abstract}
The sums of three trigonometric series with logarithmic coefficients are derived by extending an approach first utilized by Lerch. 
By applying Frullani's theorem to two of these series, two non-trivial integrals involving hyperbolic functions are evaluated in terms of the digamma function. 
\keywords{Trigonometric series \and Logarithmic coefficients \and Digamma function \and Frullani's theorem.}
\subclass{42A20 \and 42A24 \and 42A32}
\end{abstract}

\section{Introduction}

Trigonometric series~\cite{ZygmundBook,Trickovic2008} are of vital importance in mathematics and physics, especially when such series correspond to a Fourier series. 
The physical applications of Fourier series in mathematical physics abound, and often it is desirable to obtain a closed-form expression for their sum. 
One familiar example is Kummer's expression~\cite{WhittakerWatsonBook} for the logarithm of the Gamma function, which involves a series whose coefficients include a logarithmic term. 
In his elegant paper~\cite{Lerch1897} ``$\ddot{\mathrm{U}}$ber eine Formel aus der Theorie der Gammafunction'', 
Lerch derives the sum of another particular trigonometric series with a logarithmic coefficient. 
The result is expressed as follows: if $0<\nu<1$,
\begin{eqnarray}
\sum_{k=2}^{\infty}\log\left(1-\frac{1}{k^{2}}\right)\cos\left(2k\nu\pi\right)&=&2\sin^{2}\left(\nu\pi\right)\left(\psi\left(1\right)-\psi\left(\nu\right)-\log\left(\pi\right)\right)
\nonumber\\&&-\frac{\pi}{2}\sin\left(2\nu\pi\right)-\log\left(2\right).\label{eq:Lerch1}
\end{eqnarray}
By implementing similar methodology, the sums of three other trigonometric series with different logarithmic coefficients than that of Eq.~\eqref{eq:Lerch1} can be derived. 
Thus, it is the purpose of this paper to present a derivation of the following results: 
\begin{thm}
\label{eq:MainTheorem}
If $0<\nu<1$,
\begin{eqnarray}
&&\quad\sum_{k=2}^{\infty}\log\left(\frac{k-1}{k+1}\right)\sin\left(2k\nu\pi\right) \nonumber\\=&&   \sin\left(2\nu\pi\right)\left(\log\left(4\pi\right)-\psi\left(1\right)+\psi\left(\nu\right)\right)+\pi\cos^{2}\left(\nu\pi\right).\label{eq:Series2}\\
&&\sum_{k=2}^{\infty}\log\left(\frac{2k-1}{2k+1}\right)\sin\left(2k\nu\pi\right) \nonumber\\=&& \frac{1}{2}\sin\left(\nu\pi\right)\left(\psi\left(\frac{\nu}{2}\right)-\psi\left(\frac{\nu+1}{2}\right)\right)+\frac{\pi}{2}+\log\left(3\right)\sin\left(2\nu\pi\right).\label{eq:Series3}\\
&&\sum_{k=2}^{\infty}\log\left(\frac{4\left(k^{2}-1\right)}{4k^{2}-1}\right)\cos\left(2k\nu\pi\right) \nonumber\\= && \frac{1}{2}\cos\left(\nu\pi\right)\left(\psi\left(\frac{\nu+1}{2}\right)
-\psi\left(\frac{\nu}{2}\right)-4\log\left(2\right)\cos\left(\nu\pi\right)\right)-\frac{\pi}{2}\sin\left(2\nu\pi\right)\nonumber \\
 &  & +\cos\left(2\nu\pi\right)\left(\log\left(\frac{3\pi}{2}\right)-\psi\left(1\right)+\psi\left(\nu\right)\right).\label{eq:Series4}
\end{eqnarray}
\end{thm}

Further identities can be obtained upon combining Eq.~\eqref{eq:Lerch1} with Eq.~\eqref{eq:Series4} or Eq.~\eqref{eq:Series2} with Eq.~\eqref{eq:Series3}.
In Ref.~\cite{Lerch1897}, Lerch applies Frullani's theorem to Eq.~\eqref{eq:Lerch1} to obtain a closed-form expression for the integral of a particular theta function.
Applying Frullani's theorem to the series appearing in Eq.~\eqref{eq:Series2} and Eq.~\eqref{eq:Series3} leads to two series that can be explicitly summed, and,  as will be shown, this produces two non-trivial integration identities. 
The series given here also provide novel representations of the digamma function, and they can be used to obtain certain integrals of the digamma function~\cite{Connon}. 
The derivation of these results is now presented. 

\section{Derivation of results }
\subsection{An intermediate lemma}
The starting point is Eq.~(1) of Ref.~\cite{Lerch1897}, which states:
\begin{lem}
If $0<x<1$, $0<\nu<1$, 
\begin{equation}
\sum_{n=0}^{\infty}\frac{e^{-2x\pi i\left(\nu+n\right)}}{\nu+n}=-\log\left(2\pi x\right)+\Gamma^{\prime}\left(1\right)-\frac{\Gamma^{\prime}\left(\nu\right)}{\Gamma\left(\nu\right)}-\frac{i\pi}{2}-\sideset{}{^{\prime}}\sum_{k=-\infty}^{\infty}e^{2k\nu\pi i}\log\left(\frac{x+k}{k}\right).\label{eq:Lerch2}
\end{equation}
\end{lem}
Here, the prime notation appearing in the series means that the term $k=0$ is excluded. 
For convenience, logarithmic derivatives of the Gamma function shall be written in terms of the digamma function: $\psi\text{\ensuremath{\left(z\right)}}\equiv\Gamma^{\prime}\left(z\right)/\Gamma\left(z\right)$.
In Ref.~\cite{Lerch1895}, Lerch proved the result in Eq.~\eqref{eq:Lerch2} by taking the derivative of the series on the left-hand side (with $ix$ replaced by $x$) and then summing the resulting series.  
Upon integrating the result, Lerch then verified the Lemma. For completeness, an alternative derivation of Eq.~\eqref{eq:Lerch2} is provided here. 
\begin{proof}
To begin, the following result (due to Kronecker) can be obtained from page 123 of Ref.~\cite{WhittakerWatsonBook}.
If $0<\nu<1$,
\begin{equation}
\lim_{N\rightarrow\infty}\sideset{}{^{\prime}}\sum_{k=-N}^{N}\frac{e^{2k\nu\pi i}}{k+y}=2\pi i\frac{e^{-2\pi i\nu y}}{1-e^{-2\pi iy}}-\frac{1}{y}.\label{eq:WhittakerWatson1}
\end{equation}
Now integrate both sides of Eq.~\eqref{eq:WhittakerWatson1} with respect to $y$, with the domain of integration between $y=\frac{\epsilon}{2\pi}$ and $y=x$, where $0<x<1$, and at the end of the calculation take the limit $\epsilon\rightarrow0$.
Using term-by-term integration, the left-hand side becomes 
\begin{equation*}
\lim_{\epsilon\rightarrow0}\int_{\frac{\epsilon}{2\pi}}^{x}dy\sideset{}{^{\prime}}\sum_{k=-\infty}^{\infty}\frac{e^{2k\nu\pi i}}{k+y}=\sideset{}{^{\prime}}\sum_{k=-\infty}^{\infty}e^{2k\nu\pi i}\log\left(\frac{x+k}{k}\right).
\end{equation*}
Perform the same procedure on the right-hand side of Eq.~\eqref{eq:WhittakerWatson1}, and then substitute $t=2\pi iy$ to obtain 
\begin{eqnarray*}
\lim_{\epsilon\rightarrow0}\int_{\frac{\epsilon}{2\pi}}^{x}dy\left(2\pi i\frac{e^{-2\pi i\nu y}}{1-e^{-2\pi iy}}-\frac{1}{y}\right) & = & \lim_{\epsilon\rightarrow0}\int_{i\epsilon}^{2\pi ix}dt\left(\frac{e^{-\nu t}}{1-e^{-t}}-\frac{1}{t}\right) \\
 & = & \lim_{\epsilon\rightarrow0}\left[\int_{i\epsilon}^{2\pi ix}dt\frac{e^{-\nu t}}{1-e^{-t}}-\log\left(\frac{2\pi x}{\epsilon}\right)\right].
\end{eqnarray*}
The remaining integral is now evaluated by converting it to a contour integral in the complex plane. 
Indeed, consider 
\begin{equation*}
\widetilde{I}=\oint_{\mathcal{C}}dz\frac{e^{-\nu z}}{1-e^{-z}}.
\end{equation*}
Here, the closed contour $\mathcal{C}$ is the concatenation of the contours $C_{1},C_{2},C_{3}$, and $C_{\epsilon}$, where $C_{1}$ is the contour along $\left(0,i\epsilon\right)$ to $\left(0,2\pi ix\right)$,
$C_{2}$ is the contour along $\left(0,2\pi ix\right)$ to $\left(\infty,0\right)$, $C_{3}$ is the contour along $\left(\infty,0\right)$ to $\left(\epsilon,0\right)$, 
and $C_{\epsilon}$ is the quarter-circle contour skirting round from $\left(\epsilon,0\right)$ to $\left(0,i\epsilon\right)$. 
By Cauchy's theorem $\widetilde{I}=0$. The integral $I_{C_{\epsilon}}$ can be computed using the quarter-circle contour, and in the limit $\epsilon\rightarrow0$ the result is $I_{C_{\epsilon}}=\frac{2\pi i}{4}$. 
Combining these results together then implies 
\begin{equation}
\lim_{\epsilon\rightarrow0}\int_{i\epsilon}^{2\pi ix}dt\frac{e^{-\nu t}}{1-e^{-t}}=\lim_{\epsilon\rightarrow0}\left[\int_{\epsilon}^{\infty}dt\frac{e^{-\nu t}}{1-e^{-t}}-\int_{2\pi ix}^{\infty}dt\frac{e^{-\nu t}}{1-e^{-t}}-\frac{i\pi}{2}\right].\label{eq:Int1}
\end{equation}
The first integral on the right-hand side of Eq.~\eqref{eq:Int1} is 
\begin{eqnarray*}
\lim_{\epsilon\rightarrow0}\int_{\epsilon}^{\infty}dt\frac{e^{-\nu t}}{1-e^{-t}} & = & \lim_{\epsilon\rightarrow0}\int_{\epsilon}^{\infty}dt\frac{e^{-\nu t}-e^{-t}+e^{-t}}{1-e^{-t}} \\
 & = & \int_{0}^{\infty}dt\frac{e^{-\nu t}-e^{-t}}{1-e^{-t}}+\lim_{\epsilon\rightarrow0}\int_{\epsilon}^{\infty}dt\frac{e^{-t}}{1-e^{-t}} \\
 & = & \psi\left(1\right)-\psi\left(\nu\right)-\lim_{\epsilon\rightarrow0}\log\left(\epsilon\right).
\end{eqnarray*}
In the third line, the integral representation of the digamma function~(page 247 of Ref.~\cite{WhittakerWatsonBook}) has been used:
\begin{equation*}
\psi\left(z\right)=\int_{0}^{\infty}dt\left(\frac{e^{-t}}{t}-\frac{e^{-zt}}{1-e^{-t}}\right),\quad \mathrm{Re}(z)>0.
\end{equation*}

The second integral on the right-hand side of Eq.~\eqref{eq:Int1} is computed by expanding the denominator as a geometric series. 
Thus,
\begin{equation*}
\int_{2\pi ix}^{\infty}dt\frac{e^{-\nu t}}{1-e^{-t}}=\sum_{n=0}^{\infty}\frac{e^{-2\pi ix\left(\nu+n\right)}}{\nu+n}.
\end{equation*}
Combining the previous results then produces 
\begin{eqnarray*}
\sideset{}{^{\prime}}\sum_{k=-\infty}^{\infty}e^{2k\nu\pi i}\log\left(\frac{x+k}{k}\right) & = & -\log\left(2\pi x\right)+\psi\left(1\right)-\psi\left(\nu\right)\nonumber\\
  &&-\frac{i\pi}{2}-\sum_{n=0}^{\infty}\frac{e^{-2\pi ix\left(\nu+n\right)}}{\nu+n}.
\end{eqnarray*}
Rearranging this expression then gives the desired result in Eq.~\eqref{eq:Lerch2}.
\end{proof}

\subsection{Proof of theorem 1}
We now proceed to prove the results in Eqs.~\eqref{eq:Series2}-\eqref{eq:Series4}. 
\begin{proof}
The first step is to set $x=\frac{1}{2}$ in Eq.~\eqref{eq:Lerch2} and multiply both sides by $e^{i\nu\pi}$. 
After simplifying, the expression obtained is
\begin{eqnarray}
&&\sideset{-}{^{\prime}}\sum_{k=-\infty}^{\infty}e^{\left(2k+1\right)i\nu\pi}\log\left(\frac{2k+1}{2k}\right)\nonumber\\ & = & e^{i\nu\pi}\left(\log\left(\pi\right)-\psi\left(1\right)+\psi\left(\nu\right)+\frac{i\pi}{2}\right)
+\sum_{n=0}^{\infty}\frac{\left(-1\right)^{n}}{\nu+n}.\label{eq:Lerch3}
\end{eqnarray}
The sum on the right-hand side of Eq.~\eqref{eq:Lerch3} can be explicitly computed using the digamma function~\cite{WhittakerWatsonBook,BromwichBook}, which leads to
\begin{eqnarray}
\sideset{-}{^{\prime}}\sum_{k=-\infty}^{\infty}e^{\left(2k+1\right)i\nu\pi}\log\left(\frac{2k+1}{2k}\right) & = & e^{i\nu\pi}\left(\log\left(\pi\right)-\psi\left(1\right)+\psi\left(\nu\right)+\frac{i\pi}{2}\right)\nonumber\\
&&+\frac{1}{2}\left(\psi\left(\frac{\nu+1}{2}\right)-\psi\left(\frac{\nu}{2}\right)\right).\label{eq:Lerch4}
\end{eqnarray}
By equating the real and imaginary parts of each side of this equation, two results are obtained. 
In the case of the imaginary parts, after rewriting the $k$-summation over a single sum from $k=1$ to $k=\infty$, 
the following expression is produced
\begin{eqnarray}
\sum_{k=1}^{\infty}\log\left(\frac{k}{k+1}\right)\sin\left(\left(2k+1\right)\nu\pi\right)&=&\sin\left(\nu\pi\right)\left(\log\left(2\pi\right)-\psi\left(1\right)+\psi\left(\nu\right)\right)\nonumber\\
&&+\frac{\pi}{2}\cos\left(\nu\pi\right).\label{eq:Imag1}
\end{eqnarray}
Now multiply this result by $2\sin\left(\nu\pi\right)$, then use the identity $2\sin\left(a\right)\sin\left(b\right)=\cos\left(a-b\right)-\cos\left(a+b\right)$,
and finally relabel indices to obtain
\begin{eqnarray*}
&& 2\sin^{2}\left(\nu\pi\right)\left(\log\left(2\pi\right)-\psi\left(1\right)+\psi\left(\nu\right)\right)+\frac{\pi}{2}\sin\left(2\nu\pi\right) \\ & = & \sum_{k=1}^{\infty}\log\left(\frac{k}{k+1}\right)\left[\cos\left(2k\nu\pi\right)-\cos\left(\left(2k+2\right)\nu\pi\right)\right] \\
 & = & \sum_{k=2}^{\infty}\log\left(\frac{k}{k+1}\frac{k}{k-1}\right)\cos\left(2k\nu\pi\right)+\log\left(\frac{1}{2}\right)\cos\left(2\nu\pi\right) \\
 & = & -\sum_{k=2}^{\infty}\log\left(1-\frac{1}{k^{2}}\right)\cos\left(2k\nu\pi\right)-\log\left(2\right)\cos\left(2\nu\pi\right).
\end{eqnarray*}
Simplifying this equation then gives the result in Eq.~\eqref{eq:Lerch1}.

It is possible to obtain three more results from Eq.~\eqref{eq:Lerch4}, either by multiplying the imaginary part by $2\cos\left(\nu\pi\right)$, or by taking the real part and then multiplying by $2\sin\left(\nu\pi\right)$ or $2\cos\left(\nu\pi\right)$. 
Indeed, upon multiplying Eq.~\eqref{eq:Imag1} by $2\cos\left(\nu\pi\right)$, then using the identity $2\sin\left(a\right)\cos\left(b\right)=\sin\left(a+b\right)+\sin\left(a-b\right)$,
and after finally relabelling indices, the result obtained is 
\begin{eqnarray*}
&&\sin\left(2\nu\pi\right)\left(\log\left(2\pi\right)-\psi\left(1\right)+\psi\left(\nu\right)\right)+\pi\cos^{2}\left(\nu\pi\right)\nonumber\\ & = & \sum_{k=1}^{\infty}\log\left(\frac{k}{k+1}\right)\left[\sin\left(\left(2k+2\right)\nu\pi\right)+\sin\left(2k\nu\pi\right)\right]\nonumber \\
 & = & \sum_{k=2}^{\infty}\log\left(\frac{k-1}{k+1}\right)\sin\left(2k\nu\pi\right)+\log\left(\frac{1}{2}\right)\sin\left(2\nu\pi\right).
\end{eqnarray*}
Rearranging this equation then produces the result in Eq.~\eqref{eq:Series2}.

Returning now to Eq.~\eqref{eq:Lerch4}, by taking the real part of this equation, and rewriting the $k$-summation over a single sum from $k=1$ to $k=\infty$, the result obtained is
\begin{eqnarray}
&&-\sum_{k=1}^{\infty}\log\left(\frac{\left(2k+1\right)^{2}}{2^{2}k\left(k+1\right)}\right)\cos\left(\left(2k+1\right)\nu\pi\right)\nonumber\\
&=&\cos\left(\nu\pi\right)\left(\log\left(\frac{\pi}{2}\right)-\psi\left(1\right)+\psi\left(\nu\right)\right)\nonumber\\&&
-\frac{\pi}{2}\sin\left(\nu\pi\right)+\frac{1}{2}\left(\psi\left(\frac{\nu+1}{2}\right)-\psi\left(\frac{\nu}{2}\right)\right).\label{eq:Real1}
\end{eqnarray}
Multiply this result by $2\sin\left(\nu\pi\right)$, then again use the identity $2\sin\left(a\right)\cos\left(b\right)=\sin\left(a+b\right)+\sin\left(a-b\right)$,
and finally relabel indices to obtain
\begin{eqnarray*}
 &  & \sin\left(2\nu\pi\right)\left(\log\left(\frac{\pi}{2}\right)-\psi\left(1\right)+\psi\left(\nu\right)\right)-\pi\sin^{2}\left(\nu\pi\right) \\
 &&+\sin\left(\nu\pi\right)\left(\psi\left(\frac{\nu+1}{2}\right)-\psi\left(\frac{\nu}{2}\right)\right) \\
 & = & -\sum_{k=1}^{\infty}\log\left(\frac{\left(2k+1\right)^{2}}{2^{2}k\left(k+1\right)}\right)\left[\sin\left(\left(2k+2\right)\nu\pi\right)-\sin\left(2k\nu\pi\right)\right] \\
 & = & -\sum_{k=2}^{\infty}\log\left(\left(\frac{2k-1}{2k+1}\right)^{2}\left(\frac{k+1}{k-1}\right)\right)\sin\left(2k\nu\pi\right)+\log\left(\frac{9}{8}\right)\sin\left(2\nu\pi\right). 
\end{eqnarray*}
After simplifying this expression and using the result in Eq.~\eqref{eq:Series2}, the result in Eq.~\eqref{eq:Series3} is obtained.

Now multiply Eq.~\eqref{eq:Real1} by $2\cos\left(\nu\pi\right)$, then use the identity $2\cos\left(a\right)\cos\left(b\right)=\cos\left(a+b\right)+\cos\left(a-b\right)$,
and finally relabel indices to obtain 
\begin{eqnarray}
 &  & 2\cos^{2}\left(\nu\pi\right)\left(\log\left(\frac{\pi}{2}\right)-\psi\left(1\right)+\psi\left(\nu\right)\right)-\frac{\pi}{2}\sin\left(2\nu\pi\right) \nonumber\\
 &&+\cos\left(\nu\pi\right)\left(\psi\left(\frac{\nu+1}{2}\right)-\psi\left(\frac{\nu}{2}\right)\right) \nonumber\\
 & = & -\sum_{k=1}^{\infty}\log\left(\frac{\left(2k+1\right)^{2}}{2^{2}k\left(k+1\right)}\right)\left[\cos\left(\left(2k+2\right)\nu\pi\right)+\cos\left(2k\nu\pi\right)\right] \nonumber\\
 & = & -\sum_{k=2}^{\infty}\left(\frac{\left(4k^{2}-1\right)^{2}}{16k^{2}\left(k^{2}-1\right)}\right)\cos\left(2k\nu\pi\right)-\log\left(\frac{9}{8}\right)\cos\left(2\nu\pi\right). \label{eq:Series4a}
\end{eqnarray}
After simplifying this expression and using the result in Eq.~\eqref{eq:Lerch1}, the result in Eq.~\eqref{eq:Series4} is obtained.
\end{proof}
This completes the derivation of the results in Eqs.~\eqref{eq:Lerch1}-\eqref{eq:Series4}; the first result was derived in Ref.~\cite{Lerch1897}, whereas the other three results have been derived in this paper.

\section{Applications }

\subsection{Applying Frullani's theorem}

Frullani's theorem (see page 479 of Ref.~\cite{BromwichBook} or page 656 of Ref.~\cite{HardyBook7}; for further discussion of a more general class of Frullanian integrals
see page 195 of Ref.~\cite{HardyBook5}) leads to the following result (see pages 116-117 of Ref.~\cite{WhittakerWatsonBook}); if $a,b>0$,
\begin{equation*}
\log\left(\frac{b}{a}\right)=\int_{0}^{\infty}\frac{dx}{x}\left(e^{-ax}-e^{-bx}\right).
\end{equation*}
Using this result, two integration identities shall be proved. 
\begin{thm}
If $0<\nu<1$,
\begin{eqnarray}
&&\int_{0}^{\infty}\frac{dx}{x}\sinh\left(x\right)e^{-x}\left(\frac{e^{-x}-2\cos\left(2\nu\pi\right)}{\cosh\left(x\right)-\cos\left(2\nu\pi\right)}\right)\nonumber\\
&=&\log\left(4\pi\right)-\psi\left(1\right)+\psi\left(\nu\right)+\frac{\pi}{2}\cot\left(\nu\pi\right).\label{eq:IntThm1}
\end{eqnarray}
\end{thm}
\begin{proof}
Consider the result in Eq.~\eqref{eq:Series2}. 
Applying Frullani's theorem to the logarithm in the summand leads to 
\begin{eqnarray}
 &    & \sin\left(2\nu\pi\right)\left(\log\left(4\pi\right)-\psi\left(1\right)+\psi\left(\nu\right)\right)+\pi\cos^{2}\left(\nu\pi\right)\nonumber\\
 & = & \sum_{k=2}^{\infty}\int_{0}^{\infty}\frac{dx}{x}\left(e^{-\left(k+1\right)x}-e^{-\left(k-1\right)x}\right)\sin\left(2k\nu\pi\right)\nonumber \\
 & = & -\int_{0}^{\infty}\frac{dx}{x}2\sinh\left(x\right)\sum_{k=2}^{\infty}e^{-kx}\sin\left(2k\nu\pi\right).\label{eq:Frullani1}
\end{eqnarray}
The series above can be exactly summed, as shown forthwith.
\begin{eqnarray*}
\sum_{k=2}^{\infty}e^{-kx}\sin\left(2k\nu\pi\right) & = & \text{Im}\sum_{k=2}^{\infty}e^{-k\left(x-2i\nu\pi\right)} \\
 & = & \text{Im}\frac{\exp\left(4i\nu\pi-2x\right)}{1-\exp\left(2i\nu\pi-x\right)} \\
 & = & \frac{1}{2}e^{-x}\sin\left(2\nu\pi\right)\left(\frac{2\cos\left(2\nu\pi\right)-e^{-x}}{\cosh\left(x\right)-\cos\left(2\nu\pi\right)}\right).
\end{eqnarray*}
Inserting this identity into Eq.~\eqref{eq:Frullani1} and then simplifying leads to the result in Eq.~\eqref{eq:IntThm1}.
\end{proof}

\begin{thm}
If $0<\nu<1$,
\begin{eqnarray}
&&\int_{0}^{\infty}\frac{dx}{x}\sinh\left(x\right)e^{-2x}\left(\frac{e^{-2x}-2\cos\left(2\nu\pi\right)}{\cosh\left(2x\right)-\cos\left(2\nu\pi\right)}\right)\nonumber\\
&=&\frac{\pi}{2}\text{\ensuremath{\mathrm{cosec}}}\left(2\nu\pi\right)+\log\left(3\right)
+\frac{1}{4}\sec\left(\nu\pi\right)\left(\psi\left(\frac{\nu}{2}\right)-\psi\left(\frac{\nu+1}{2}\right)\right).\label{eq:IntThm2}
\end{eqnarray}
\end{thm}
\begin{proof}
Applying Frullani's theorem to the series in Eq.~\eqref{eq:Series3} gives
\begin{eqnarray}
&&\frac{1}{2}\sin\left(\nu\pi\right)\left(\psi\left(\frac{\nu}{2}\right)-\psi\left(\frac{\nu+1}{2}\right)\right)+\frac{\pi}{2}+\log\left(3\right)\sin\left(2\nu\pi\right)\nonumber\\ 
&=& \sum_{k=2}^{\infty}\int_{0}^{\infty}\frac{dx}{x}\left(e^{-\left(2k+1\right)x}-e^{-\left(2k-1\right)x}\right)\sin\left(2k\nu\pi\right)\nonumber \\
&=& -\int_{0}^{\infty}\frac{dx}{x}2\sinh\left(x\right)\sum_{k=2}^{\infty}e^{-2kx}\sin\left(2k\nu\pi\right).\label{eq:Frullani2}
\end{eqnarray}
The sum can be computed as shown previously, and the result is 
\begin{equation*}
\sum_{k=2}^{\infty}e^{-2kx}\sin\left(2k\nu\pi\right)=\frac{1}{2}e^{-2x}\sin\left(2\nu\pi\right)\left(\frac{2\cos\left(2\nu\pi\right)-e^{-2x}}{\cosh\left(2x\right)-\cos\left(2\nu\pi\right)}\right).
\end{equation*}
Inserting this identity into Eq.~\eqref{eq:Frullani2} and then simplifying leads to the result in Eq.~\eqref{eq:IntThm2}.
\end{proof}
These two integration results are non-trivial since neither Maple nor Mathematica appear to be able to calculate the closed-form expressions.
Nevertheless, numerical evaluation of these integrals for arbitrary values of $\nu\in(0,1)$ agrees with the analytical results.

\subsection{Particular results for the cases $\nu=0$ and $\nu=1$}

As discussed in Ref.~\cite{Lerch1897}, the validity of the result in Eq.~\eqref{eq:Lerch1} can be extended from $0<\nu<1$ to $0\leq\nu\leq1$.
For $\nu=0$ and $\nu=1$, this leads to the result 
\begin{equation}
\label{eq:InfSeries0}
\sum_{k=2}^{\infty}\log\left(1-\frac{1}{k^{2}}\right)=-\log\left(2\right).
\end{equation}

For the series in Eq.~\eqref{eq:Series4}, the logarithm tends to $-3/(4k^2)$, as $k\rightarrow\infty$, and thus the series is absolutely convergent. 
Hence, the validity of the result in Eq.~\eqref{eq:Series4} can be extended from $0<\nu<1$ to $0\leq\nu\leq1$. 
Setting $\nu=0$ (or $\nu=1)$ then leads to the summation identity
\begin{equation}
\label{eq:InfSeries1}
\sum_{k=2}^{\infty}\log\left(\frac{4\left(k^{2}-1\right)}{4k^{2}-1}\right)=\log\left(\frac{3\pi}{16}\right).
\end{equation} 
The proof of this results is as follows. 
\begin{proof}Consider the $\nu=0$ limit of Eq.~\eqref{eq:Series4}: 
\begin{eqnarray*}
 & & \sum_{k=2}^{\infty}\log\left(\frac{4\left(k^{2}-1\right)}{4k^{2}-1}\right)\\
 & = & \lim_{\nu\rightarrow0}\biggl[\frac{1}{2}\cos\left(\nu\pi\right)\left(\psi\left(\frac{\nu+1}{2}\right)-\psi\left(\frac{\nu}{2}\right)-4\log\left(2\right)\cos\left(\nu\pi\right)\right)-\frac{\pi}{2}\sin\left(2\nu\pi\right)\\
 &  & +\cos\left(2\nu\pi\right)\left(\log\left(\frac{3\pi}{2}\right)-\psi\left(1\right)+\psi\left(\nu\right)\right)\biggr]\\
 & = & \frac{1}{2}\left(\psi\left(\frac{1}{2}\right)-\psi\left(1\right)-4\log\left(2\right)\right) + \log\left(\frac{3\pi}{2}\right)\\
 & = & \log\left(\frac{3\pi}{16}\right).
\end{eqnarray*}
\end{proof}
In the last step we have used~\cite{BromwichBook}: $\psi\left(1\right)=-\gamma$ and $\psi\left(\frac{1}{2}\right)=-\gamma-2\log\left(2\right)$, where $\gamma$ is the Euler-Mascheroni constant. 
The same result is obtained if the limit $\nu\rightarrow1$ is taken in Eq.~\eqref{eq:Series4}.
As another example, setting $\nu=0$ (the same result is obtained if $\nu=1$) in Eq.~\eqref{eq:Series4a} gives
\begin{equation}
\label{eq:InfSeries2}
\sum_{k=2}^{\infty}\log\left(\frac{\left(4k^{2}-1\right)^{2}}{16k^{2}\left(k^{2}-1\right)}\right)=\log\left(\frac{128}{9\pi^{2}}\right).
\end{equation}
\begin{proof}
Consider the $\nu=0$ limit of Eq.~\eqref{eq:Series4a} :
\begin{eqnarray*}
 & & \sum_{k=2}^{\infty}\log\left(\frac{\left(4k^{2}-1\right)^{2}}{16k^{2}\left(k^{2}-1\right)}\right)\\
 & = & -\lim_{\nu=0}\biggl[\cos\left(\nu\pi\right)\left(\psi\left(\frac{\nu+1}{2}\right)-\psi\left(\frac{\nu}{2}\right)\right)+\log\left(\frac{9}{8}\right)\cos\left(2\nu\pi\right)\\
 &  &+2\cos^{2}\left(\nu\pi\right)\left(\log\left(\frac{\pi}{2}\right)-\psi\left(1\right)+\psi\left(\nu\right)\right) -\frac{\pi}{2}\sin\left(2\nu\pi\right)\biggr] \\
 & = & \psi\left(1\right)-\psi\left(\frac{1}{2}\right)-\log\left(\frac{9}{8}\right)-2\log\left(\frac{\pi}{2}\right) \\
 & = & \log\left(\frac{128}{9\pi^{2}}\right).
\end{eqnarray*}
\end{proof}

\subsection{Infinite products}

The infinite series results in Eqs.~\eqref{eq:InfSeries0}-\eqref{eq:InfSeries2} are particular cases of a more general result~\cite{Chamberland2013} involving infinite products. 
Indeed, as an example, if we let $k=m+2$, then the result in Eq.~\eqref{eq:InfSeries2} can be expressed as the following infinite product
\begin{eqnarray}
\label{eq:InfProd}
\prod_{m=0}^{\infty}\frac{\left[\left(m+\frac{5}{2}\right)\left(m+\frac{3}{2}\right)\right]^2}{\left(m+3\right)\left(m+2\right)^2\left(m+1\right)}=\frac{128}{9\pi^2}.
\end{eqnarray}
According to theorem 1.1 of Ref.~\cite{Chamberland2013}, the infinite product above evaluates to $\Gamma(3)\left(\Gamma(2)\right)^2\Gamma(1)/\left[\Gamma(5/2)\Gamma(3/2)\right]^2=2^7/(9\pi^2)$,
in agreement with Eq.~\eqref{eq:InfProd}. The results in Eqs.~\eqref{eq:InfSeries0}-\eqref{eq:InfProd} can also be confirmed by using Maple or Mathematica. 

Since the exponential mapping $x\mapsto\exp(x)$ is a continuous function and a group isomorphism from $\left(\mathbb{R},+\right)$ to $\left(\mathbb{R}_{>0},\times\right)$
all of the results in Theorem~\ref{eq:MainTheorem} can be equivalently expressed as infinite products, by taking the exponential of both sides of the equations. 
As one non-trivial example, Wallis' formula~\cite{BromwichBook} for $\pi$ can be derived. The proof is as follows. Set $\nu=\frac{1}{4}$ in Eq.~\eqref{eq:Series2} to obtain
\begin{eqnarray*}
\quad\sum_{k=2}^{\infty}\log\left(\frac{k-1}{k+1}\right)\sin\left(\frac{k\pi}{2}\right) &&= \sum_{l=1}^{\infty}\left(-1\right)^{l}\log\left(\frac{l}{l+1}\right)\nonumber\\
&&= \log\left(4\pi\right)-\psi\left(1\right)+\psi\left(\frac{1}{4}\right)+\frac{\pi}{2}\\
&&=\log\left(\frac{\pi}{2}\right).
\end{eqnarray*}
In the last step we have used~\cite{BromwichBook}: $\psi\left(\frac{1}{4}\right)=-\frac{\pi}{2}-3\log\left(2\right)+\psi\left(1\right)$.
Taking the exponential of both sides of the equation above then gives Wallis' formula:
\begin{equation*}
\prod_{l=1}^{\infty}\left[\left(\frac{l}{l+1}\right)^{\left(-1\right)^{l}}\right]=\prod_{k=1}^{\infty}\left(\frac{2k}{2k-1}\cdot\frac{2k}{2k+1}\right)=\frac{\pi}{2}.
\end{equation*}

\begin{acknowledgements}
This research was performed while the author was supported by the Department of Physics and Theoretical Physics Institute at the University of Alberta.
\end{acknowledgements}

\section*{Compliance with ethical standards}
\textbf{Conflict of interest}: The author declares that he has no conflict of interest.\\
\textbf{Ethical approval}: This article does not contain any studies with human participants or animals performed by the author.\\
\textbf{Informed consent}: For this type of study informed consent was not required.

\end{document}